\documentclass[11pt,reqno,a4paper]{amsart}
\usepackage{amssymb}
\usepackage{amsfonts}
\usepackage{amsmath, amsthm, amssymb}
\usepackage[english]{babel}
\usepackage[mathcal]{eucal}
\usepackage{fancyhdr}
\usepackage{pdfsync}
\usepackage{graphicx}
\usepackage{placeins}
\usepackage{pdflscape}
\usepackage{lipsum, float}
\usepackage{listings}
\usepackage{setspace}

\allowdisplaybreaks

\title[]{On the asymptotic behavior of the contaminated sample mean}
\author{Ben Berckmoes}
\author{Geert Molenberghs}

\subjclass[2000]{}
\thanks{Ben Berckmoes is post doctoral fellow at the Fund for Scientific Research of Flanders (FWO)}
\thanks{Geert Molenberghs gratefully acknowledges financial support from the IAP research network  \#P7/06 of the Belgian Government (Belgian Science Policy)}

\DeclareMathOperator*{\Var}{\textrm{Var}}
\begin{document}

\maketitle

\newtheorem{pro}{Proposition}[section]
\newtheorem{lem}[pro]{Lemma}
\newtheorem{thm}[pro]{Theorem}
\newtheorem{de}[pro]{Definition}
\newtheorem{co}[pro]{Comment}
\newtheorem{no}[pro]{Notation}
\newtheorem{vb}{Example}
\newtheorem{vbn}[pro]{Examples}
\newtheorem{gev}[pro]{Corollary}
\newtheorem{rem}[pro]{Remark}

\begin{abstract}
An observation of a cumulative distribution function $F$ with finite variance is said to be contaminated according to the inflated variance model if it has a large probability of coming from the original target distribution $F$, but a small probability of coming from a contaminating distribution that has the same mean and shape as
$F$, though a larger variance. It is well known that in the presence of data contamination, the ordinary sample mean looses many of its good properties, making it preferable to use more robust estimators. It is insightful to see to what extent an intuitive estimator such as the sample mean becomes less favorable in a contaminated setting. In this paper, we investigate under which conditions the sample mean, based on a finite number of independent
observations of $F$ which are contaminated according to the inflated variance model, is a valid estimator for the mean of $F$. In particular, we examine to what extent this estimator is weakly consistent for the mean of $F$ and asymptotically normal. As classical central limit theory is generally inaccurate to cope with the asymptotic normality in this setting, we invoke more general approximate central limit theory as developed by Berckmoes, Lowen, and Van Casteren (2013). Our theoretical results are
illustrated by a specific example and a simulation study.
\end{abstract}

\section{Introduction}

Suppose that we are given a finite number of independent observations $X_1,\ldots,X_n$ of a cumulative distribution $F$ on the real line with mean $\mu$ and finite variance. It is well known that the sample mean $\overline{X}_n = \frac{1}{n}\sum_{k=1}^n X_k$ is an accurate estimator in the sense that it is consistent for $\mu$, asymptotically normal, and under broad conditions also efficient.

Now assume that there is an underlying mechanism causing each observation to be contaminated according to the inflated variance model (Titterington, Smith, and Makov 1985:108). That is, instead of each $X_k$ having distribution $F$, there is a large probability that $X_k$ comes from the distribution $F$, but a small probability
that it comes from a contaminating distribution $F(\frac{\cdot}{\sigma_k})$, $\sigma_k \geq \text{\upshape Var}(F)$, which has the same shape and mean as $F$, though a larger variance. It is well known that in this contaminated setting, the ordinary sample mean becomes less reliable, and more robust estimators should be used (Huber and Ronchetti 2009). It is insightful to show to what extent the good properties of the sample mean are lost in this setting.

In this paper, we investigate under what conditions the sample mean in this contaminated setting remains (weakly) consistent for $\mu$ and asymptotically normal. It will turn out that the weak consistency can be easily established under a fairly weak condition using Chebyshev's inequality (Theorem \ref{thm:WeakCon}) and that the asymptotic normality can be established using classical central limit theory if the sequence of contaminating variances $(\sigma_k)_k$ can be controlled sufficiently (Theorem \ref{thm:CEst}).

However, if no sufficient control over the sequence $(\sigma_k)_k$ is possible, we end up with an interesting class of settings in which classical central limit theory is inaccurate to describe the asymptotic behavior of the sample mean. This is due to the fact that in this case the question of whether the sample mean is asymptotically normal should be answered, rather than in a dichotomous (yes/no) way, in a continuous fashion. Instead we will use the approximate generalization of classical central limit theory, as developed by Berckmoes et al. (2013), to produce a number between $0$ and $1$, which is interpretable as an upper bound for a canonical index measuring how far the sample mean deviates from being asymptotically normal (Theorem \ref{thm:QEst}). We thus get a gradation in which we are closer to asymptotic normality in some cases, and further away from it in others. This will be made visible by QQ-plots coming from a simulation study

The paper is structured as follows. In section 2 the formal framework in which we will work and the notation
we will use is sketched. The key results of approximate central limit theory developed by Berckmoes et al. (2013) are explained in section 3. Section 4 is the core of this paper. Here we give the theoretical results concerning the asymptotic properties of the sample mean in the contaminated setting. A specific example and a simulation study are given in section 5. Finally, in section 6 we formulate some open questions for further research.

\section{Formal framework}\label{sec:intro}

Let $F$ be a cumulative distribution function on the real line with
\begin{equation*}
\int_{-\infty}^\infty x dF(x) = 0
\end{equation*}
and
\begin{equation*}
\int_{- \infty}^\infty x^2 dF(x) = 1.
\end{equation*}
Fix $\mu \in \mathbb{R}$ and  let $X_1$, $X_2$, $\ldots$, $X_k$, $\ldots$ be independent observations of $F(\cdot - \mu)$ which are contaminated according to the inflated variance model (Titterington et al. 1985:108), that is
\begin{equation*}
X_k \sim (1 - p_k) F(\cdot - \mu) + p_k F\left(\frac{\cdot - \mu}{\sigma_k}\right),
\end{equation*}
where $p_k \in \left[0,1\right]$ and $\sigma_k \in \left[1, \infty\right[$. Observe that
\begin{equation*}
\mathbb{E}[X_k] = \mu
\end{equation*}
and
\begin{equation*}
\Var[X_k] = (1 - p_k) + p_k \sigma_k^2.
\end{equation*}
Now define the sample mean in the usual way as
\begin{equation*}
\overline{X}_n = \frac{1}{n} \sum_{k=1}^n X_k.
\end{equation*}
Notice that 
\begin{equation*}
\mathbb{E}[\overline{X}_n] = \mu
\end{equation*}
and
\begin{equation*}
\Var[\overline{X}_n] = \left(\frac{s_n}{n}\right)^2,
\end{equation*} 
where
\begin{equation*}
s_n^2 = \sum_{k=1}^n [(1-p_k) + p_k \sigma_k^2].
\end{equation*}
Also,
\begin{equation}
s_n^2 \geq n,\label{eq:sngeqn}
\end{equation}
because $\sigma_k^2 \geq 1$ for all $k$. In this paper, we investigate to what extent the estimator $\overline{X}_n$ is weakly consistent for $\mu$ in the sense that
\begin{equation}
\overline{X}_n \stackrel{\mathbb{P}}{\rightarrow} \mu\label{WeakCon}
\end{equation}
and asymptotically normal in the sense that
\begin{equation}
\frac{n}{s_n} \left(\overline{X}_n - \mu\right)  \stackrel{w}{\rightarrow} N(0,1).\label{AsymNor}
\end{equation}
Notice that in the uncontaminated case where $\sigma_k = 1$ for all $k$, the weak law of large numbers implies the truth of (\ref{WeakCon}) and the central limit theorem implies the validity of (\ref{AsymNor}). For our study of the asymptotic normality, we fall back on approximate central limit theory as developed by Berckmoes et al. (2013). We briefly recall the basics of this theory in the next section. 

\section{Approximate central limit theory}

By a standard triangular array (STA)  we mean a triangular array of real square integrable random variables 
\begin{displaymath}
\begin{array}{cccc} 
\xi_{1,1} &  &  \\
\xi_{2,1} & \xi_{2,2} & \\
\xi_{3,1} & \xi_{3,2} & \xi_{3,3} \\
 & \vdots &
\end{array}
 \end{displaymath}
with the following properties:
\begin{enumerate}
 \item $\forall n : \xi_{n,1}, \ldots, \xi_{n,n} \textrm{ are independent,}$
 \item $\forall n, k : \mathbb{E}\left[\xi_{n,k}\right] = 0,$
 \item $\forall n : \sum_{k=1}^{n} \mathbb{E}\left[\xi_{n,k}^2\right] = 1$.
\end{enumerate}

We say that an STA $\{\xi_{n,k}\}$ satisfies Lindeberg's condition iff 
\begin{equation*}
\forall \epsilon > 0 : \lim_{n \rightarrow \infty} \sum_{k=1}^n \mathbb{E}\left[\xi_{n,k}^2 ; \left|\xi_{n,k}\right| \geq \epsilon\right] = 0
\end{equation*}
and that it satisfies Feller's condition iff
\begin{equation*}
\lim_{n \rightarrow \infty} \max_{k=1}^n \mathbb{E}\left[\xi_{n,k}^2\right] = 0.
\end{equation*}
It is well known (and readily verified) that Lindeberg's condition is strictly stronger than Feller's condition. 

The importance of the notions explained above is reflected by the following key result in classical central limit theory (Feller 1971).

\begin{thm}[Central Limit Theorem (CLT)]\label{CLT}
Consider, for $\xi \sim N(0,1)$ and $\{\xi_{n,k}\}$ an STA, the following assertions:
\begin{enumerate}
\item The weak convergence relation $\sum_{k=1}^n \xi_{n,k} \stackrel{w}{\rightarrow} \xi$ holds.
\item The STA $\{\xi_{n,k}\}$ satisfies Lindeberg's condition.
\end{enumerate}
Then assertion (2) implies assertion (1) and both assertions are equivalent if $\{\xi_{n,k}\}$ satisfies Feller's condition.
\end{thm}

Now suppose that we are given an STA $\{\xi_{n,k}\}$ which satisfies Feller's condition, but fails to satisfy Lindeberg's condition. Then we infer from classical central limit theory (Theorem \ref{CLT}) that the row-wise sums of $\{\xi_{n,k}\}$ fail to be asymptotically normal. However, inspired by approach theory, a topological theory pioneered by Lowen (2015) (the details of which are not needed for a proper understanding of this paper), we could ask the following question. How far does $\sum_{k=1}^n \xi_{n,k}$ deviate from $\xi$ if $n$ gets large?

In order to formalize this question, recall that the Kolmogorov distance between random variables $\eta$ and $\eta^\prime$ is given by
\begin{equation}
K(\eta,\eta^\prime) = \sup_{x \in \mathbb{R}} \left|\mathbb{P}[\eta \leq x] - \mathbb{P}[\eta^\prime \leq x]\right|.\label{def:Kol}
\end{equation}
It is well known that for a continuously distributed random variable $\eta$ and an arbitrary sequence of random variables $(\eta_n)_n$ the following are equivalent:
\begin{enumerate}
	\item $\eta_n \stackrel{w}{\rightarrow} \eta,$
	\item $K(\eta,\eta_n) \rightarrow 0.$
\end{enumerate}
Thus, even for an STA $\{\xi_{n,k}\}$ for which the sequence $\left(\sum_{k=1}^n \xi_{n,k}\right)_n$ fails to converge weakly to $\xi$, it still makes sense to consider the number 
\begin{equation}
\limsup_{n \rightarrow \infty} K\left(\xi, \sum_{k=1}^n \xi_{n,k}\right),\label{def:limop}
\end{equation}
which takes values between 0 and 1 and measures in a precise sense how far the sequence of row-wise sums deviates from being asymptotically normal. In the language of approach theory, expression (\ref{def:limop}) is referred to as a limit operator or an index of convergence (Berckmoes, Lowen, and Van Casteren 2011; Lowen 2015). Notice that $\limsup_{n \rightarrow \infty} K\left(\xi, \sum_{k=1}^n \xi_{n,k}\right) = 0$ if and only if $\sum_{k=1}^n \xi_{n,k} \stackrel{w}{\rightarrow} \xi$.

For an arbitrary STA $\{\xi_{n,k}\}$ it also makes sense to introduce the number
\begin{equation}
\textrm{\upshape{Lin}}\left(\{\xi_{n,k}\}\right) = \sup_{\epsilon > 0} \limsup_{n \rightarrow \infty} \sum_{k = 1}^n \mathbb{E}\left[\xi_{n,k}^2 ; \left|\xi_{n,k}\right| \geq \epsilon\right],\label{def:Lin}
\end{equation}
which lies between 0 and 1 and is a canonical index which measures how far $\{\xi_{n,k}\}$ deviates from satisfying Lindeberg's condition. Expression (\ref{def:Lin}) is called the Lindeberg index. Observe that $\textrm{\upshape{Lin}}(\{\xi_{n,k}\}) = 0$ if and only if $\{\xi_{n,k}\}$ satisfies Lindeberg's condition.

The following result, which connects the expressions (\ref{def:limop}) and (\ref{def:Lin}) for an arbitrary STA satisfying Feller's condition, lies at the heart of approximate central limit theory developed by Berckmoes et al. (2013). The proof relies on Stein's method (Barbour and Chen 2005).

\begin{thm}[Approximate Central Limit Theorem (ACLT)]\label{QCLT}
Consider $\xi \sim N(0,1)$ and $\{\xi_{n,k}\}$ an STA which satisfies Feller's condition. Then 
\begin{equation*}
\limsup_{n \rightarrow \infty} K\left(\xi,\sum_{k=1}^n \xi_{n,k}\right) \leq \textrm{\upshape{Lin}}\left(\{\xi_{n,k}\}\right).
\end{equation*}
\end{thm}

Notice that the ACLT is a generalization of the CLT which has the advantage that it can cope with STA's which fail to satisfy Lindeberg's condition. The ACLT heuristically states that if an STA is close to satisfying Lindeberg's condition, then its row-wise sums are close to being asymptotically normally distributed. 

We make use of approximate central limit theory in the next section, where we study the asymptotic behavior of the contaminated sample mean as introduced in the previous section.

\section{Consistency and asymptotic normality}

We keep the terminology and the notation from above.

The following relatively straightforward result shows that the contaminated sample mean is weakly consistent under a fairly mild condition. 

\begin{thm}\label{thm:WeakCon}
Suppose that 
\begin{equation}\label{ConA}
\lim_{n \rightarrow \infty} \frac{1}{n^2} \sum_{k = 1}^n p_k \sigma_k^2  = 0.
\end{equation}
Then
\begin{equation*}
\overline{X}_n \stackrel{\mathbb{P}}{\rightarrow} \mu.
\end{equation*}
\end{thm}

\begin{proof}
Assume without loss of generality that $\mu = 0$. For $\epsilon > 0$, Chebyshev's inequality gives
\begin{eqnarray*}
\mathbb{P}\left[\left|\overline{X}_n\right| \geq \epsilon\right] &\leq& \frac{1}{\epsilon^2} \Var[\overline{X}_n]\\
&=& \frac{1}{\epsilon^2}\left[\frac{1}{n^2} \sum_{k = 1}^n (1 - p_k) + \frac{1}{n^2} \sum_{k = 1}^n p_k \sigma_k^2\right],
\end{eqnarray*}
which easily implies that
\begin{displaymath}
\limsup_{n \rightarrow \infty} \mathbb{P}\left[\left|\overline{X}_n\right| \geq \epsilon\right] \leq \frac{1}{\epsilon^2} \limsup_{n \rightarrow \infty} \frac{1}{n^2} \sum_{k = 1}^n   p_k \sigma_k^2.
\end{displaymath}
This finishes the proof.
\end{proof}

We now turn to the asymptotic normality of $\overline{X}_n$.  It turns out that the STA $\left\{\frac{1}{s_n} \left(X_k - \mu\right)\right\}$, which is of crucial importance, satisfies Lindeberg's condition if the sequence of contaminating variances $\left(\sigma_k\right)_k$ is controllable in a sense made precise in the following theorem.

\begin{thm}\label{thm:AsymNorEasy}
Suppose that
\begin{equation}
\lim_{n \rightarrow \infty} \frac{1}{s_n^2} \max_{k=1}^n \sigma_k^2 \rightarrow 0.\label{ConB}
\end{equation}
Then the STA $\left\{\frac{1}{s_n}\left(X_{k} - \mu\right)\right\}$ satisfies Lindeberg's condition, i.e.
\begin{displaymath}
\textrm{\upshape{Lin}}\left(\left\{\frac{1}{s_n} \left(X_k - \mu\right)\right\}\right) = 0.
\end{displaymath}
\end{thm}

\begin{proof}
Assume without loss of generality that $\mu = 0$ and let $X$ be a random variable with cumulative distribution function $F$. Then, for $\epsilon > 0$,
\begin{eqnarray*}
\lefteqn{\sum_{k = 1}^n \mathbb{E}\left[\left(\frac{1}{s_n}X_k\right)^2 ; \left|\frac{1}{s_n} X_k\right| \geq \epsilon\right]}\\
&=& \frac{1}{s_n^2} \sum_{k = 1}^n (1 - p_k) \mathbb{E}\left[X^2 ; \left|X\right| \geq \epsilon s_n\right] + \frac{1}{s_n^2} \sum_{k = 1}^n p_k \sigma_k^2 \mathbb{E}\left[X^2 ; \left|X\right| \geq \frac{\epsilon s_n }{\sigma_k}\right],
\end{eqnarray*} 
which is
\begin{eqnarray*}
&\leq& \frac{1}{s_n^2} \sum_{k = 1}^n (1 - p_k) \mathbb{E}\left[X^2 ; \left|X\right| \geq \epsilon s_n \right]\\
&&+  \frac{1}{s_n^2} \sum_{k = 1}^n p_k \sigma_k^2 \mathbb{E}\left[X^2 ; \left|X\right| \geq \epsilon \sqrt{\frac{s_n^2}{\max_{k=1}^n \sigma_k^2}}\right]\\
&\leq& \mathbb{E}[X^2 ; \left|X\right| \geq \epsilon s_n] + \mathbb{E}\left[X^2 ; \left|X\right| \geq \epsilon \sqrt{\frac{s_n^2}{\max_{k=1}^n \sigma_k^2}}\right].
\end{eqnarray*}
The latter quantity converges to $0$ as $n$ tends to $\infty$ by (\ref{eq:sngeqn}) and (\ref{ConB}). This finishes the proof.
\end{proof}

\begin{rem}\label{rem:AImB}
Observe that (\ref{ConB}) implies (\ref{ConA}).
\end{rem}

Classical central limit theory (Theorem \ref{CLT}) now leads to the following result.

\begin{thm}\label{thm:CEst}
Let $\xi \sim N(0,1)$ and suppose that
\begin{equation*}
\lim_{n \rightarrow \infty} \frac{1}{s_n^2} \max_{k=1}^n \sigma_k^2 \rightarrow 0.
\end{equation*}
Then
\begin{equation*}
\frac{n}{s_n} \left(\overline{X}_n - \mu\right) \stackrel{w}{\rightarrow} \xi.
\end{equation*}
\end{thm}

\begin{proof}
Notice that the $n$-th rowwise sum of $\left\{\frac{1}{s_n}\left(X_k - \mu\right)\right\}$ coincides with $\frac{n}{s_n} \left(\overline{X}_n - \mu\right)$. Now apply Theorem \ref{thm:AsymNorEasy} and Theorem \ref{CLT}.
\end{proof}

If the sequence $(\sigma_k)_k$ cannot be controlled by condition (\ref{ConB}), then it is more appropriate to make use of approximate central limit theory as outlined in the previous section. As Feller's condition plays an important role in this theory, we start with the following characterization. 

\begin{thm}\label{thm:FelNeg}
The STA $\left\{\frac{1}{s_n} (X_k - \mu)\right\}$ satisfies Feller's condition if and only if 
\begin{equation}
\lim_{n \rightarrow \infty} \frac{1}{s_n^2} \max_{k=1}^n p_k \sigma_k^2 = 0.\label{ConC}
\end{equation}
\end{thm}

\begin{proof}
Assume without loss of generality that $\mu = 0$. Now
\begin{equation*}
\max_{k = 1}^n \mathbb{E}\left[\frac{1}{s_n^2} X_k^2\right] = \frac{1}{s_n^2} \max_{k = 1}^n (1 - p_k) + \frac{1}{s_n^2}\max_{k = 1}^n p_k \sigma_k^2,
\end{equation*}
whence, by (\ref{eq:sngeqn}),
\begin{equation*}
\limsup_{n \rightarrow \infty} \max_{k = 1}^n \mathbb{E}\left[\frac{1}{s_n^2} X_k^2\right] = \limsup_{n \rightarrow \infty} \frac{1}{s_n^2} \max_{k = 1}^n p_k \sigma_k^2.
\end{equation*}
This finishes the proof.
\end{proof}

\begin{rem}
Observe that (\ref{ConB}) implies (\ref{ConC}), which in turn implies (\ref{ConA}).
\end{rem}

Theorems \ref{thm:MainLin} and \ref{thm:MainLin2} will reveal that even in the absence of condition (\ref{ConB}), the Lindeberg index of the STA $\left\{\frac{1}{s_n} \left(X_k - \mu\right)\right\}$ can still be bounded from above. Moreover, it can be explicitly computed under a fairly easy set of conditions. We need the following lemma.

\begin{lem}\label{lem:MainLin}
Suppose that the sequence $\left(\frac{1}{n} \sum_{k = 1}^n p_k \sigma_k^2\right)_n$ is bounded and let $X$ be a random variable with cumulative distribution function $F$. Then
\begin{eqnarray*}
\lefteqn{\textrm{\upshape{Lin}}\left(\left\{\frac{1}{s_n} (X_k - \mu)\right\}\right)}\\
&=& \sup_{\gamma > 0} \sup_{\epsilon > 0} \limsup_{n \rightarrow \infty} \frac{1}{s_n^2} \sum_{k = \lceil \gamma n \rceil}^n p_k \sigma_k^2 \mathbb{E}\left[X^2 ; \left|X\right| \geq \frac{\epsilon s_n}{\sigma_k}\right],
\end{eqnarray*}
where $\lceil \cdot \rceil$ is the ceiling function.
\end{lem}

\begin{proof}
Assume without loss of generality that $\mu = 0$ and choose $K \in \mathbb{R}^+_0$ such that for all $n$
\begin{equation}
\frac{1}{n} \sum_{k = 1}^n p_k \sigma_k^2 \leq K.\label{eq:CesaroBddByK} 
\end{equation}
Next, fix $\gamma > 0$ small. Then, for $n$ large, by (\ref{eq:sngeqn}) and (\ref{eq:CesaroBddByK}),
\begin{eqnarray*}
\lefteqn{\frac{1}{s_n^2} \sum_{k = 1}^{\lceil \gamma n \rceil - 1} p_k \sigma_k^2 \mathbb{E}\left[X^2 ; \left|X\right| \geq \frac{\epsilon s_n}{\sigma_k}\right]}\\
&\leq& \gamma \frac{1}{\gamma n} \sum_{k = 1}^{\lceil \gamma n \rceil - 1} p_k \sigma_k^2\\
&\leq& \gamma \frac{1}{\lceil \gamma n \rceil - 1}  \sum_{k = 1}^{\lceil \gamma n \rceil - 1} p_k \sigma_k^2\\
&\leq& K \gamma,
\end{eqnarray*} 
whence
\begin{eqnarray*}
\lefteqn{\limsup_{n \rightarrow \infty} \frac{1}{s_n^2} \sum_{k = 1}^{n} p_k \sigma_k^2 \mathbb{E}\left[X^2 ; \left|X\right| \geq \frac{\epsilon s_n}{\sigma_k}\right]}\\
&\leq& \limsup_{n \rightarrow \infty} \frac{1}{s_n^2} \sum_{k = 1}^{\lceil \gamma n \rceil - 1} p_k \sigma_k^2 \mathbb{E}\left[X^2 ; \left|X\right| \geq \frac{\epsilon s_n}{\sigma_k}\right]\\
&&+ \limsup_{n \rightarrow \infty} \frac{1}{s_n^2} \sum_{k = \lceil \gamma n\rceil}^{n} p_k \sigma_k^2 \mathbb{E}\left[X^2 ; \left|X\right| \geq \frac{\epsilon s_n}{\sigma_k}\right]\\
&\leq& K \gamma + \limsup_{n \rightarrow \infty} \frac{1}{s_n^2} \sum_{k = \lceil \gamma n \rceil}^{n} p_k \sigma_{k}^2 \mathbb{E}\left[X^2 ; \left|X\right| \geq \frac{\epsilon s_n}{\sigma_k}\right].
\end{eqnarray*}
Thus we have shown that 
\begin{eqnarray}
\lefteqn{\limsup_{n \rightarrow \infty} \frac{1}{s_n^2} \sum_{k = 1}^{n} p_k \sigma_k^2 \mathbb{E}\left[X^2 ; \left|X\right| \geq \frac{\epsilon s_n}{\sigma_k}\right]}\label{eq:gammaenough}\\
&=& \sup_{\gamma > 0} \limsup_{n \rightarrow \infty}  \frac{1}{s_n^2} \sum_{k = \lceil \gamma n \rceil}^{n} p_k \sigma_{k}^2 \mathbb{E}\left[X^2 ; \left|X\right| \geq \frac{\epsilon s_n}{\sigma_k}\right].\nonumber
\end{eqnarray}
Now, arguing analogously as in the proof of Theorem \ref{thm:AsymNorEasy} and using (\ref{eq:gammaenough}), we get
\begin{eqnarray*}
\lefteqn{\textrm{\upshape{Lin}}\left(\left\{\frac{1}{s_n} X_k\right\}\right)}\\
&=& \sup_{\epsilon > 0} \limsup_{n \rightarrow \infty} \frac{1}{s_n^2}  \sum_{k = 1}^n \mathbb{E}\left[X^2 ; \left|X\right| \geq \frac{\epsilon s_n}{\sigma_k}\right]\\
&=& \sup_{\epsilon > 0} \sup_{\gamma > 0} \limsup_{n \rightarrow \infty}  \frac{1}{s_n^2} \sum_{k = \lceil \gamma n \rceil}^{n} p_k \sigma_{k}^2 \mathbb{E}\left[X^2 ; \left|X\right| \geq \frac{\epsilon s_n}{\sigma_k}\right]\\
&=& \sup_{\gamma > 0} \sup_{\epsilon > 0} \limsup_{n \rightarrow \infty}  \frac{1}{s_n^2} \sum_{k = \lceil \gamma n \rceil}^{n} p_k \sigma_{k}^2 \mathbb{E}\left[X^2 ; \left|X\right| \geq \frac{\epsilon s_n}{\sigma_k}\right],
\end{eqnarray*}
completing the proof.
\end{proof}

\begin{thm}\label{thm:MainLin}
The inequality
\begin{equation}
\textrm{\upshape{Lin}}\left(\left\{\frac{1}{s_n} (X_k - \mu)\right\}\right) \leq \limsup_{n \rightarrow \infty} \frac{1}{s_n^2} \sum_{k = 1}^n p_k \sigma_k^2\label{LinIneqMain}
\end{equation}
always holds. If, in addition, 
\begin{enumerate}
\item $\displaystyle{(\sigma^2_n)_n}$ is monotonically increasing,
\item $\displaystyle{\liminf_{n \rightarrow \infty} \frac{1}{n} \sigma_n^2 > 0}$,
\item $\displaystyle{\left(\frac{1}{n} \sum_{k = 1}^n p_k \sigma_k^2\right)_n}$ is bounded,
\end{enumerate}
then the inequality in (\ref{LinIneqMain}) becomes an equality. 
\end{thm}

\begin{proof}
Inequality (\ref{LinIneqMain}) is easily established by the fact that $\mathbb{E}\left[X^2\right] = 1$. Now suppose that the three additional conditions in Theorem \ref{thm:MainLin} are fulfilled. The fact that
\begin{displaymath}
\liminf_{n \rightarrow \infty} \frac{1}{n} \sigma_n^2 > 0
\end{displaymath}
allows us to choose $\delta > 0$ and $n_0$ such that for all $n \geq n_0$
\begin{equation}
\sigma_n^2 \geq \delta n.\label{eq:conliminf}
\end{equation}
Furthermore, the boundedness of $\left(\frac{1}{n} \sum_{k = 1}^n p_k \sigma_k^2\right)_n$ allows us to pick $K \in \mathbb{R}^+_0$ such that for all $n$
\begin{equation}
\frac{1}{n} \sum_{k = 1}^n p_k \sigma_k^2 \leq K.\label{eq:CesBddByK2}
\end{equation}
Now fix $\gamma > 0$ small. Then, for $n$ so large that 
\begin{equation}
\lceil \gamma n\rceil \geq n_0\label{eq:nverylarge}
\end{equation}
and for $k$ such that
\begin{equation}
\lceil \gamma n \rceil \leq k \leq n,\label{eq:kgood}
\end{equation}
we have, by (\ref{eq:nverylarge}), (\ref{eq:kgood}), (\ref{eq:CesBddByK2}), and (\ref{eq:conliminf}),
\begin{eqnarray*}
\left(\frac{s_n}{\sigma_k}\right)^2 &=& \frac{\sum_{k = 1}^n (1 - p_k) + \sum_{k = 1}^n p_k \sigma_k^2}{\sigma_k^2}\\
&\leq& \frac{\sum_{k = 1}^n (1 - p_k) + \sum_{k = 1}^n p_k \sigma_k^2}{\delta k}\\
&\leq& \frac{\sum_{k = 1}^n (1 - p_k) + \sum_{k = 1}^n p_k \sigma_k^2}{\delta \lceil \gamma n \rceil}\\
&\leq& \frac{1}{\delta \gamma} \left(\frac{1}{n} \sum_{k = 1}^n \left(1 - p_k\right) + \frac{1}{n} \sum_{k = 1}^n p_k \sigma_k^2\right)\\
&\leq& \frac{1 + K}{\delta \gamma},
\end{eqnarray*}
whence
\begin{equation*}
\mathbb{E}\left[X^2 ; \left|X\right| \geq \frac{\epsilon s_n}{\sigma_k}\right] \geq \mathbb{E}\left[X^2 ; \left|X\right| \geq \epsilon \sqrt{\frac{1 + K}{\delta \gamma}}\right],
\end{equation*}
with $X$ a random variable with cumulative distribution function $F$. In particular,
\begin{eqnarray}
\lefteqn{\sup_{\epsilon > 0} \limsup_{n \rightarrow \infty} \frac{1}{s_n^2} \sum_{k = \lceil \gamma n\rceil}^n p_k \sigma_k^2 \mathbb{E}\left[X^2 ; \left|X\right| \geq \frac{\epsilon s_n}{\sigma_k}\right]}\label{eq:BigIneqLin}\\
&\geq& \sup_{\epsilon > 0} \limsup_{n \rightarrow \infty} \frac{1}{s_n^2} \sum_{k = \lceil \gamma n\rceil}^n p_k \sigma_k^2 \mathbb{E}\left[X^2 ; \left|X\right| \geq \epsilon \sqrt{\frac{1 + K}{\delta \gamma}}\right]\nonumber\\
&=& \sup_{\epsilon > 0} \mathbb{E}\left[X^2 ; \left|X\right| \geq \epsilon \sqrt{\frac{1 + K}{\delta \gamma}}\right] \limsup_{n \rightarrow \infty} \frac{1}{s_n^2} \sum_{k = \lceil \gamma n\rceil}^n p_k \sigma_k^2\nonumber\\
&=& \limsup_{n \rightarrow \infty} \frac{1}{s_n^2} \sum_{k = \lceil \gamma n\rceil}^n p_k \sigma_k^2,\nonumber
\end{eqnarray}
where the last equality follows from the fact that $\mathbb{E}\left[X^2\right] = 1$. Combining Lemma \ref{lem:MainLin} and the inequality shown by (\ref{eq:BigIneqLin}) gives
\begin{eqnarray*}
\lefteqn{\textrm{\upshape{Lin}}\left(\left\{\frac{1}{s_n} (X_k - \mu)\right\}\right)}\label{eq:keyineqtex}\\
&=& \sup_{\gamma > 0} \sup_{\epsilon > 0} \limsup_{n \rightarrow \infty} \frac{1}{s_n^2} \sum_{k = \lceil \gamma n \rceil}^n p_k \sigma_k^2 \mathbb{E}\left[X^2 ; \left|X\right| \geq \frac{\epsilon s_n}{\sigma_k}\right]\nonumber\\
&\geq& \sup_{\gamma > 0} \limsup_{n \rightarrow \infty} \frac{1}{s_n^2} \sum_{k = \lceil \gamma n\rceil}^n p_k \sigma_k^2\nonumber\\
&=& \limsup_{n \rightarrow \infty} \frac{1}{s_n^2} \sum_{k = 1}^n p_k \sigma_k^2,\nonumber
\end{eqnarray*}
the last equality following by mimicking the proof of Lemma \ref{lem:MainLin}. This finishes the proof.
\end{proof}

\begin{thm}\label{thm:MainLin2}
Suppose that
\begin{enumerate}
\item $\displaystyle{\left(\frac{1}{n} \sum_{k = 1}^n p_k \sigma_k^2\right)_n}$ is convergent to $L \in \mathbb{R}^+$,
\item $\displaystyle{\left(\frac{1}{n}\sum_{k=1}^n p_k\right)_n}$ is convergent to 0.
\end{enumerate}
Then the inequality
\begin{equation}
\textrm{\upshape{Lin}}\left(\left\{\frac{1}{s_n} (X_k - \mu)\right\}\right) \leq \frac{L}{1 + L}\label{eq:concformlin}
\end{equation}
holds. If, in addition, 
\begin{itemize}
\item[(3)] $\displaystyle{(\sigma^2_n)_n}$ is monotonically increasing,
\item[(4)] $\displaystyle{\liminf_{n \rightarrow \infty} \frac{1}{n} \sigma_n^2 > 0}$,
\end{itemize}
then the inequality in (\ref{eq:concformlin}) becomes an equality.
\end{thm}

\begin{proof}
Theorem \ref{thm:MainLin} gives
\begin{eqnarray*}
\textrm{\upshape{Lin}}\left(\left\{\frac{1}{s_n} \left(X_k - \mu\right)\right\}\right) &\leq& \limsup_{n \rightarrow \infty} \frac{1}{s_n^2} \sum_{k = 1}^n p_k \sigma_k^2\\
&=& \limsup_{n \rightarrow \infty} \frac{\sum_{k = 1}^n p_k \sigma_k^2}{\sum_{k = 1}^n (1 - p_k) + \sum_{k = 1}^n p_k \sigma_k^2}\\
&=& \limsup_{n \rightarrow \infty} \frac{\frac{1}{n} \sum_{k = 1}^n p_k \sigma_k^2}{1 - \frac{1}{n} \sum_{k = 1}^n  p_k + \frac{1}{n}\sum_{k = 1}^n p_k \sigma_k^2}\\
&=& \frac{L}{1 + L},
\end{eqnarray*}
the last equality following from conditions (1) and (2) in Theorem \ref{thm:MainLin2}. This establishes (\ref{eq:concformlin}). If conditions (3) and (4) in Theorem \ref{thm:MainLin2} are also satisfied, then Theorem \ref{thm:MainLin} shows that the first inequality in the above calculation becomes an equality and we are done.
\end{proof}

Now approximate central limit theory (Theorem \ref{QCLT}) gives the following result. Recall that the Kolmogorov distance is given by (\ref{def:Kol}).

\begin{thm}\label{thm:QEst}
Let $\xi \sim N(0,1)$ and suppose that
\begin{enumerate}
\item $\displaystyle{\left(\frac{1}{n} \sum_{k = 1}^n p_k \sigma_k^2\right)_n}$ is convergent to $L \in \mathbb{R}^+$,
\item $\displaystyle{\left(\frac{1}{n} \sum_{k=1}^n p_k\right)_n}$ is convergent to $0$,
\item $\displaystyle{\left(\frac{1}{s_n^2}\max_{k = 1}^n p_k \sigma_k^2\right)_n}$ is convergent to $0$.
\end{enumerate}
Then
\begin{equation}
\limsup_{n \rightarrow \infty} K\left(\xi,\frac{n}{s_n}\left( \overline{X}_n - \mu\right)\right) \leq \frac{L}{1 + L}.\label{ineq:Main}
\end{equation}
\end{thm}

\begin{proof}
Theorem \ref{thm:FelNeg} is applicable to conclude that the STA $\left\{\frac{1}{s_n} \left(X_k - \mu\right)\right\}$ satisfies Feller's condition. Furthermore, Theorem \ref{thm:MainLin2} reveals that the Lindeberg index of this STA is bounded from above by $\frac{L}{1+L}$. Finally, the $n$-th row-wise sum of this STA coinciding with $\frac{n}{s_n} \left(\overline{X}_n - \mu\right)$, it suffices to apply Theorem \ref{QCLT}.
\end{proof}

We wish to make the following final reflection. If, in addition to the conditions formulated in Theorem \ref{thm:QEst}, $(\sigma_n)_n$ increases monotonically and $\liminf_{n \rightarrow \infty} \frac{1}{n} \sigma_n > 0$, then, by Theorem \ref{thm:MainLin2}, $\textrm{\upshape{Lin}}\left(\left\{\frac{1}{s_n} (X_k - \mu)\right\}\right) = \frac{L}{1 + L}$. Thus, if $L \neq 0$, classical central limit theory (Theorem \ref{CLT}) leads to the conclusion that the estimator $\overline{X}_n$ fails to be asymptotically normal in the sense that the sequence $\left(\frac{n}{s_n} \left(\overline{X}_n - \mu\right)\right)_n$  does not converge weakly to $\xi$. However, inequality (\ref{ineq:Main}), derived from more general approximate central limit theory (Theorem \ref{QCLT}), shows that $\overline{X}_n$ is still close to being asymptotically normal when $L$ is small. 

We empirically demonstrate these ideas in the next section through an example and a simulation study.

\section{Example and simulation study} 

We keep the terminology and the notation of the previous sections. 

In the following theorem, we apply the results obtained in the previous section to a specific choice for $p_k$ and $\sigma_k^2$. Recall that we say that $\overline{X}_n$ is weakly consistent (WC) for $\mu$ if (\ref{WeakCon}) holds and asymptotically normal (AN) if (\ref{AsymNor}) holds.

\begin{thm}\label{thm:Example}
Let
\begin{equation*}
p_k = p k^{-a} \textrm{ with } p \in \left]0,1\right[ \textrm{ and } a \in \left]0,\infty\right[
\end{equation*}
and
\begin{equation*}
\sigma_k^2 = s^2 k^b \textrm{ with } s \in \left]1,\infty\right[ \textrm{ and } b \in \left]0,\infty\right[.
\end{equation*}
Then the following assertions are true.
\begin{enumerate}
 \item If $b < 1$, then $\overline{X}_n$ is WC for $\mu$ and AN.
 \item If $b \geq 1$ and $a > b$, then $\overline{X}_n$ is WC for $\mu$ and AN.
 \item If $b \geq 1$ and $a = b$, then $\overline{X}_n$ is WC for $\mu$, but fails to be AN. However,
 \begin{equation*}
 \limsup_{n \rightarrow \infty} K\left(\xi,\frac{n}{s_n} \left(\overline{X}_n - \mu\right)\right) \leq \frac{p s^2}{1 + p s^2}.
 \end{equation*}
 \end{enumerate}
\end{thm}

\begin{proof}[Proof of Theorem \ref{thm:Example}]
Firstly, suppose that $b < 1$. Now, by (\ref{eq:sngeqn}),
\begin{equation*}
\frac{1}{s_n^2} \max_{k=1}^n \sigma_k^2 = \frac{n^b}{s_n^2} \leq n^{b - 1}, 
\end{equation*}
which clearly converges to $0$ as $n$ tends to $\infty$. Thus condition (\ref{ConB}) is satisfied, which allows us to conclude from Theorem \ref{thm:CEst} that $\overline{X}_n$ is AN. Also, Remark \ref{rem:AImB} shows that condition (\ref{ConA}) holds, whence we infer from Theorem \ref{thm:WeakCon} that $\overline{X}_n$ is WC for $\mu$. This establishes the first assertion.

Next, consider the case where $b \geq 1$ and $a > b$. Then the sequence
\begin{equation*}
p_k \sigma_k^2 = p s^2 k^{b - a}
\end{equation*}
converges to $0$ as $k$ tends to $\infty$, whence 
\begin{equation*}
\lim_{n \rightarrow \infty} \frac{1}{n} \sum_{k = 1}^n p_k \sigma_k^2 = 0.
\end{equation*}
Now it easily follows from Theorem \ref{thm:WeakCon} that $\overline{X}_n$ is WC for $\mu$ and from Theorem \ref{thm:QEst} that
\begin{equation*}
\limsup_{n \rightarrow \infty}K\left(\xi,\frac{n}{s_n} \left(\overline{X}_n - \mu\right)\right) = 0.
\end{equation*}
Put otherwise, $\overline{X}_n$ is AN and the second assertion is proved. 

Finally, let $b \geq 1$ and $a = b$. Then 
\begin{equation*}
\frac{1}{n} \sum_{k = 1}^n p_k \sigma_k^2 = p s^2.
\end{equation*}
Then, by Theorem \ref{thm:WeakCon}, $\overline{X}_n$ is WC for $\mu$. Furthermore, by Theorem \ref{thm:FelNeg}, Feller's condition is satisfied, and, by Theorem \ref{thm:MainLin2}, the Lindeberg index is $\frac{ps^2}{1 + ps^2}$. Thus, by Theorem \ref{CLT}, $\overline{X}_n$ fails to be AN. However, by Theorem \ref{thm:QEst}, the desired inequality in the third assertion holds.
\end{proof}

\begin{figure}[!htb]
  \centering
    \includegraphics[scale=0.7]{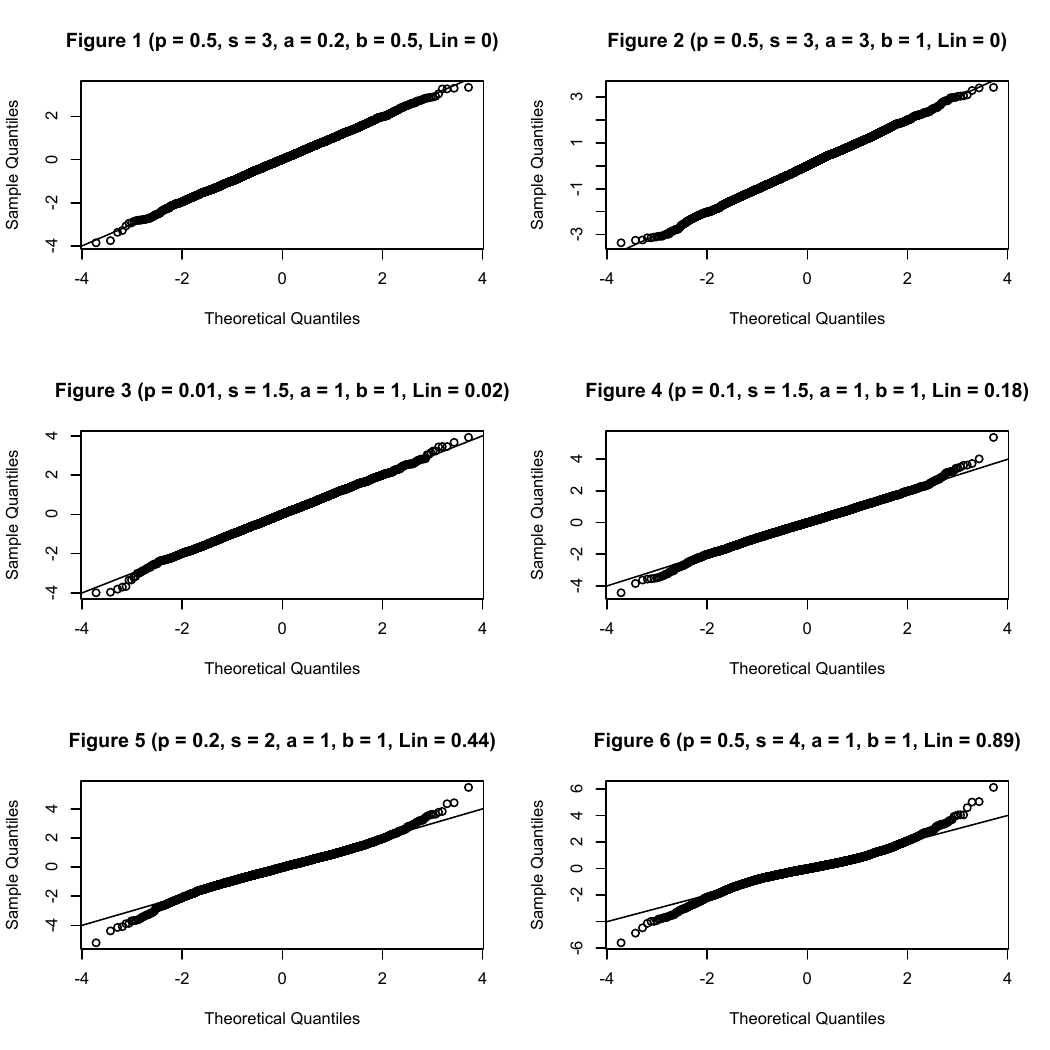}
\end{figure}

In order to illustrate Theorem \ref{thm:Example}, we have conducted a simulation study with the following setup. For specific instances of $p, s, a, b$ we have created an empirical cdf $\mathcal{E}$ for $\frac{\overline{X}_n - \mathbb{E}[\overline{X}_n]}{\sqrt{\Var(\overline{X}_n)}}$ with sample size $n = 1000$. In each case the empirical cdf was based on 5000 simulations. We have tested for asymptotic normality by creating a QQ-plot the graph of which contains bullets with coordinates $(\Phi^{-1}(t),\mathcal{E}^{-1}(t))$, where $\Phi$ is the cdf of a standard normal distribution, $\mathcal{E}$ is the empirical cdf, and $t$ runs over a specific grid from $0$ to $1$. If a bullet $(\Phi^{-1}(t),\mathcal{E}^{-1}(t))$ is close to the line $y = x$, then $\Phi^{-1}(t) \approx \mathcal{E}^{-1}(t)$, whence $\mathcal{E}(\Phi^{-1}(t)) \approx t = \Phi(\Phi^{-1}(t))$. Thus on each QQ-plot we have also added the graph of the line $y = x$. To each figure we have added the value of the Lindeberg index governing the asymptotic normality of the sample mean. Recall that the Lindeberg index always takes values between 0 and 1.

The following conclusions can be drawn from this study.

If $b < 1$, then the first assertion in Theorem \ref{thm:Example} states that - even if $p$ and $s$ are large and $a$ is below $b$ - the sample mean is asymptotically normal because the Lindeberg index is 0. This is confirmed by Figure 1.

If $b \geq 1$ and $a > b$, then the second assertion in Theorem \ref{thm:Example} states that - even if $p$ and $s$ are large - the sample mean is asymptotically normal because the Lindeberg index is 0. This is confirmed by Figure 2.

If $b \geq 1$ and $a = b$, then the third assertion in Theorem \ref{thm:Example} provides an upper bound for a canonical measure of the asymptotic normality of the sample mean because the Lindeberg index is $\frac{p s^2}{1 + p s^2}$. The larger the Lindeberg index, the more deviation from asymptotic normality could be seen. This is confirmed by Figures 3, 4 and 5.

\section{Open questions}

We formulate some open questions which could be a source for future research.

{\bf Question 1.} Theorem \ref{thm:Example} does not handle the case where $b \geq 1$ and $a < b$. Assume without loss of generality that $\mu = 0$. Then, arguing analogously as in the proof of Theorem \ref{thm:AsymNorEasy}, we easily see that 
 \begin{equation*}
\textrm{\upshape{Lin}}\left(\left\{\frac{1}{s_n} X_k\right\}\right)
= \sup_{\epsilon > 0} \limsup_{n \rightarrow \infty} \frac{1}{s_n^2}  \sum_{k = 1}^n \mathbb{E}\left[X^2 ; \left|X\right| \geq \frac{\epsilon s_n}{\sigma_k}\right],
\end{equation*}
$X$ being a random variable with cumulative distribution function $F$ and
\begin{equation*}
\sigma_k^2 = s^2 k^b
\end{equation*}
and
\begin{equation*}
s_n^2 = n - p \sum_{k = 1}^n  k^{- a} + p s^2 \sum_{k = 1}^n k^{b - a}.
\end{equation*}
It would be of interest to examine the existence of a more explicit formula for the Lindeberg index in this case. Also, the weak consistency should be investigated.\\

{\bf Question 2.} Strictly speaking, inequality (\ref{ineq:Main}) only shows that the Lindeberg index is an upper bound for a natural index measuring the asymptotic normality of the sample mean. This allows us to draw the conclusion that the sample mean is close to being asymptotically normal when the Lindeberg index is small, but we cannot say anything about what happens when the Lindeberg index is large. However, our simulation study empirically reveals that when the Lindeberg index gets larger, the sample mean tends to deviate more from asymptotic normality. It would be of interest to establish a useful lower bound for $\limsup_n K\left(\frac{n}{s_n} \left(\overline{X}_n - \mu\right) \rightarrow \xi\right)$ in terms of the Lindeberg index, which serves as a theoretical underpinning of this observation. General lower bounds of this type have been obtained by Berckmoes et al. (2013), but they are so unsharp that they do not have the power to predict what we have seen in our simulation study.

\end{document}